\documentclass{amsart}
\usepackage{amsmath, amsthm, amstext, amsfonts, amssymb, mathtools}
\usepackage{array} 
\usepackage{graphicx}
\usepackage[hidelinks]{hyperref}
\usepackage[capitalize]{cleveref}
\usepackage{xcolor}

\usepackage{colonequals}

\usepackage{tikz}
\usetikzlibrary{cd,calc}
\usetikzlibrary{decorations.pathreplacing}
\usetikzlibrary{knots, hobby, intersections}
\tikzset{axis line style/.style={thin, gray, -stealth}}
\numberwithin{equation}{section}

\theoremstyle{plain}
\newtheorem{theorem}{\bf Theorem}[section]
\newtheorem{lemma}[theorem]{\bf Lemma}
\newtheorem{proposition}[theorem]{\bf Proposition}

\theoremstyle{definition}
\newtheorem{definition}[theorem]{Definition}

\theoremstyle{remark}
\newtheorem{remark}[theorem]{Remark}

\newcolumntype{L}{>{$}l<{$}} 

\newcommand{\Mod}[1]{\ (\mathrm{mod}\,#1)}

\newcommand{\tsigma}{\tilde{\sigma}}

\newcommand{\Z}{\mathbb{Z}}

\newcommand{\C}{\mathbb{C}}
\newcommand{\Q}{\mathbb{Q}}

\newcommand{\PP}{\mathbb{P}} 

\let\sl\relax 
\DeclareMathOperator{\sl}{\mathfrak{sl}}

\DeclareMathOperator*{\res}{res}

\newcommand{\qbinom}{\genfrac{[}{]}{0pt}{}}
\newcommand{\sm}[1]{\{#1\}}

\DeclarePairedDelimiter{\floor}{\lfloor}{\rfloor}
\DeclarePairedDelimiter{\abs}{\lvert}{\rvert}

\newcommand{\lr}[1]{(\!(#1)\!)}

\newcommand{\legendre}[2]{\genfrac{(}{)}{}{}{#1}{#2}}

\newcommand{\Zhat}{\widehat{Z}}
\DeclareMathOperator{\spc}{spin^c}

\DeclareMathOperator{\vol}{vol}

\newcommand{\iu}{\text{i}}

\title{Inverted Habiro series and its residues}
\author{Josef Svoboda}
\date{\today}

\begin{document}

\begin{abstract}
    We study the Gukov--Manolescu (GM) series of knots and the inverted Habiro series (IHS) proposed by S. Park. We give a new formula for IHS in terms of coefficients of the GM series and truncated theta functions. We prove a multiplication formula for IHS, constructing a natural ring, in analogy with work of Habiro. We study the residues of IHS and apply them to Dehn surgery formulas. We also give a curious relation between the asymptotics of the GM series at roots of unity and the Kashaev invariant.
\end{abstract}

\maketitle

\section{Introduction}

\subsection{The GM series and inverted Habiro series}
Gukov and Manolescu conjectured the existence of a new quantum knot invariant, the two-variable series $F_K(x,q)$ of a knot $K$ \cite{GM19}. The motivation was to construct a knot complement counterpart of the invariant $\Zhat(Y;q)$ of closed negative-definite plumbed 3-manifolds \cite{GPPV20}. They formulated certain desirable properties of $F_K(x,q)$, e.g. integrality of its coefficients and a relation to the Melvin--Morton--Rozansky expansion of the colored Jones polynomial \cite{Melvin1995, ROZANSKY19981, BarNatan1996}. They proposed a surgery formula, which connects $\Zhat(Y;q)$ with $F_K(x,q)$, and defined $F_K(x,q)$ for certain plumbed knots, in particular for torus knots.

Park in \cite{Park20, Park21} extended the definition of $F_K(x,q)$ to knots that are closures of homogeneous braids\footnote{The construction is more general, it works for a certain class of knots which Park called `nice knots', including all fibered knots under 10 crossings. See \cite{ParkThesis} for details. Park also defined $F_K$ for twist knots, providing examples of knots where the coefficients of $F_K$ are infinite $q$-series.}, using Verma modules of the quantum algebra $U_q(\sl_2)$ and his inverted state sum. Park also proposed a new algebraic object closely related to $F_K(x,q)$, the \emph{inverted Habiro series}, which is the main subject of this paper. 

We denote the inverted Habiro series by $P_K(x,q)$ to distinguish it from $F_K(x,q)$. Both series appear in a celebrated formula of Ramanujan \cite[eq. 1.1]{Andrews1979}:
\begin{multline}\label{eq:ram_intro}
	\sum_{k \geq 0} \frac{q^k}{(1-x)\prod_{i=1}^k (1-x q^i)(1-x^{-1} q^i)} = \\
	 \sum_{k = 0}^\infty (-1)^k x^{3k}q^{k(3k+1)/2}(1-x^2q^{2k+1}) + \frac{\sum_{k = 0}^\infty (-1)^k x^{2k+1}q^{k(k+1)/2}}{(1-x)\prod_{i=1}^\infty (1-xq^i)(1-x^{-1}q^i)}.
\end{multline}

The left-hand side of \eqref{eq:ram_intro} is the inverted Habiro series $P_{3_1^r}(x,q)$ of the right-handed trefoil knot $3_1^r$. On the right-hand side, we have two terms: the first one is the series $F_{3_1^r}(x,q)$. The second term was not explicitly present in Park's work, as he was focusing on the power expansion of $P_K(x,q)$ at $x=0$ (defined as the sum of Taylor expansions of each term).
However, it's importance is evident---it detects the residues of the left-hand side at $x=q^n$ for $n \in \Z$, so it describes the behavior of $P_K(x,q)$ far from zero.

In general, $F_K(x,q)$ is a formal power series of the form\footnote{We use the positive part of the original $F_K$ invariant and our convention differs from Park's by a factor of $x^{1/2}$.}
\[
    F_K(x,q)=\sum_{n \geq 0} f_n(q) x^n,
\]
where $f_n(q) = f_{K;n}(q) \in \Z\lr{q}$ are Laurent series (finite polynomials for homogeneous braid knots).  
The \emph{inverted Habiro series} $P_K(x,q)$ is a two-variable formal expression of the form
\begin{equation}\label{eq:ih_intro}
    P_K(x,q) =  (x^{-1}-1)  \sum_{k \geq 0} \frac{a_{-k-1}(q)}{\prod_{i=0}^{k}(x+x^{-1}-q^i-q^{-i})}
\end{equation}
with $a_{-k-1}(q) \in \Z\lr{q}$.
Park imposed the term-by-term expansion of $P_K(x)$ at $x=0$ to be equal to $F_K(x,q)$. This condition determines the  \emph{inverted Habiro coefficients} $a_{-k-1}(q)=a_{K;-k-1}(q)$ uniquely in terms of the coefficients $f_n(q)$ and vice versa (see \cref{prop:ih_fk}).

The terminology was motivated by Park's observation that for simple knots, the sequence $(a_{-k-1}(q))_{k \geq 0}$ naturally extends the sequence of the usual Habiro coefficients $(a_n(q))_{n \geq 0}$ (see \eqref{eq:hab} below) to negative $n=-k-1$, via the quantum $C$-polynomial recurrence \cite{Garoufalidis2006}.

We may ask whether there is an analogous formula to \eqref{eq:ram_intro} for knots other than $3_1^r$. As the series $F_K(x,q)$ has usually zero radius of convergence for generic $q$, one cannot hope for an identity between meromorphic functions involving $F_K(x,q)$. 

Warnaar derived \eqref{eq:ram_intro} and several generalizations using his beautiful $q$-series formula \cite[Thm.\ 1.5]{Warnaar2003}:
\begin{equation}\label{eq:warnaar_intro}
        1 + \sum_{n=1}^\infty (-1)^nq^{\binom{n}{2}} (a^n + b^n) = (q,a,b;q)_\infty \sum_{n=0}^\infty \frac{(ab/q;q)_{2n}q^n}{(q,a,b,ab;q)_n}.
	\end{equation}

We use \eqref{eq:warnaar_intro} to obtain a closed formula for $P_K(x,q)$ in terms of the coefficients $f_n(q)$. Recall the Jacobi theta function:
\[
    \theta(x,q) = \sum_{n \in \Z} (-1)^n x^nq^{n(n-1)/2} = 
\prod_{n =0}^\infty (1-q^{n+1})(1-xq^n)(1-x^{-1}q^{n+1}).
    \]
Consider also truncated theta functions $\theta_j(x,q)$ defined as
    \begin{equation*}
        \theta_i(x,q) = (-1)^i q^{\binom{i+1}{2}} \left(1 + \sum_{n=1}^\infty (-1)^nq^{\binom{n+1}{2}+n i} (x^n+x^{-n})\right).
    \end{equation*}

Finally, we need a condition on the growth of the coefficients $a_n(q)$ that ensures the convergence of $P_K(x)$ as a $q$-series:
\begin{definition}
	A sequence $(a_n)_{n<0}$  in $\Z\lr{q}$ satisfies the \emph{lower bound condition}, if for all $n<0$, we have
	\[
		\delta(a_n) \geq -\frac{n(n+3)}{2} + C,
	\]
	with some constant $C$. Here $\delta(a_n)$ denotes the minimal $q$-power in $a_n(q)$.
\end{definition}

\begin{theorem}\label{thm:theta_intro}
	Assuming the lower bound condition, we have
	\begin{equation*}
		P_K(x,q)=\frac{1}{\theta(x,q)} \sum_{i \geq 0} f_i(q) \theta_i(x,q).
	\end{equation*}
\end{theorem}

This theorem sheds some light on the structure of the inverted Habiro series:
\begin{itemize}
    \item It explains why $P_K(x)$ can often be expanded as a $q$-series or interpreted as a meromorphic function: The coefficients $f_i$ of $F_K(x,q)$ typically contain increasingly negative $q$-powers, e.g. $q^{-\floor{j^2/4}}$ for the figure-eight knot $4_1$. The functions $\theta_i(x,q)$ cancel out the negative powers. In this sense, the inverted Habiro series can be interpreted as a `regularization' of the GM series.
    \item It shows that the theta function is responsible for the poles of $P_K(x,q)$ at $x=q^j$ for $j \in \Z$. In particular, we obtain a formula for their residues in terms of the coefficients $f_i(q)$.
\end{itemize}

We have a computational evidence\footnote{Thanks to P. Orland, T. Sanders-O'Court and L.S.M. Suárez.} that the lower bound condition holds true for homogeneous braid knots  up to 13 crossings. However, the condition fails for fibered knots $8_{21}, 9_{44}, 10_{133},10_{137}, \dots$ For those knots, the inverted Habiro series cannot be interpreted as a $q$-series or as a meromorphic function.

The growth of the degrees of $a_{-k-1}$ and $f_i$ (and hence the failure of the lower bound condition) is related to the boundary slopes of the knots, in the spirit of the Slope Conjecture for the colored Jones polynomial \cite[Conj. 1]{Garoufalidis2011slopes}. This relation will be studied in detail in \cite{OSSS25}.


\subsection{Ring of inverted Habiro series}

What is the origin of the inverted Habiro series in quantum algebra? To address this question, we first recall the universal quantum invariant \cite{Lawrence1988,Lawrence1999} and its relation to Melvin--Morton--Rozansky expansion \cite{Melvin1995, ROZANSKY19981, BarNatan1996}, following Habiro \cite{Habiro2007}:

Let $U_h(\sl_2)$ be the $h$-adic version of the quantized $\sl_2$ algebra (see, e.g., \cite{Kassel1995}). The universal $U_h(\sl_2)$-invariant $J_K$ of a knot $K$ is a central element in $U_h(\sl_2)$. Habiro proved that $J_K$ lies in a certain subring of $U_h(\sl_2)$, whose center can be identified with a completed polynomial ring
\[
    \Lambda = \varprojlim_n \Z[q^{\pm 1}][y]/(\sigma_n),
\]
where 
\[
\sigma_n =\prod_{i=1}^n (y+2-q^i-q^{-i}). 
\]
This implies Habiro's theorem \cite[Thm.\ 4.5]{Habiro2007} that $J_K$ can be uniquely expressed in the form 
\begin{equation}\label{eq:hab}
    J_K = \sum_{n=0}^\infty a_n \sigma_n,
\end{equation}
where the Habiro coefficients $a_n$ satisfy $a_n \in \Z[q^{\pm 1}]$. Here $y = x+x^{-1}-2$, where $x$ is the usual Alexander variable.

The ring $\Lambda$ admits maps to various simpler rings, under which $J_K$ specializes, e.g., to colored Jones polynomials or ADO polynomials \cite{Willetts2022,Beliakova2021}. A natural map $\Lambda \to \Z[[q-1,y]]$ given by expanding $\sigma_n$ into a power series near $y=0$ and $q=1$, maps $J_K$ to $R_{K}^{y \to 0}$, the Rozansky's integral version of the Melvin--Morton expansion. The superscript emphasizes that it is a power series expansion near $y=0$ (corresponding to $x=1$). According to Melvin--Morton conjecture, the series $R_{K}^{y \to 0}$ can be lifted to an element 
\begin{equation}\label{eq:mmr_rat}
R_K = \sum_{n=0}^\infty \frac{\rho_{n}(y)}{\nabla_K(y)^{2n+1}}(q-1)^n \in \Z[y,\frac{1}{\nabla_K(y)}][[q-1]]
\end{equation}
as was proved in \cite{BarNatan1996, ROZANSKY19981}. Here $\nabla_K(y) \in \Z[y]$ is a suitably normalized\footnote{The variable $y$ is square of the usual variable $z$, e.g., we have $\nabla_{4_1}(y)=1-y$.} Alexander--Conway polynomial of the knot $K$.

Let $R_{K}^{y \to \infty} \in \Z[[y^{-1},q-1]]$ be the expansion of \eqref{eq:mmr_rat} in $y^{-1}$, as opposed to $y$. The series $F_K(x,q)$ is a (conjectural) lift of $R_{K}^{y \to \infty}$ to $\Z\lr{q}[[x]]$ \cite[Conj. 1.5]{GM19}.
Inverted Habiro series $P_K$ is an intermediate object between $R_{K}^{y \to \infty}$ and $F_K$, similar in spirit to the universal invariant $J_K$. It is therefore natural to expect the existence of a ring $\Omega$ analogous to $\Lambda$, with the following properties:

\begin{enumerate}
    \item $\Omega$ is an algebra over a suitable base ring $B$ containing $\Z[q^{\pm 1}]$.
    \item $P_K \in \Omega$ for any knot $K$.
    
    \item $\Omega$ is (topologically) generated by the elements $\sigma_{-k-1}, k = 0,1,2,\dots$, the coefficients of $a_{-k-1}$ in \eqref{eq:ih_intro} (see also \eqref{eq:sigmas}).

    \item $\Omega$ is the center of a suitable integral form of the algebra $U_h(\sl_2)$.
\end{enumerate}

In (1), we choose $B = \Z\lr{q}$. This ring may be too large---the ring $B$ should carry an extension of the automorphism of $\Z[q^{\pm 1}]$ sending $q$ to $q^{-1}$, corresponding to taking mirrors of knots.

For (2), we consider inverted Habiro series that satisfy the lower bound condition. This is probably not necessary, and it should not be required for the general theory. The advantage is that we may consider the $q$-series expansions, i.e., $\Omega$ admits a natural map to $\Q(x)\lr{q}$, given by sending $\sigma_{-k-1}$ to its expansion at $q=0$. We can also study the residues of elements of $\Omega$, which was the original motivation of this paper. 

Part (3) is the foundation of our construction of $\Omega$. We establish a multiplication formula for Laurent expansions of (rational functions) $\sigma_{-k-1}$ at $x=0$ (\cref{prop:mult}). Then we define an abstract ring generated by (symbols) $\sigma_{-k-1}$ with the same multiplication constants. 

The last property (4) would be the most interesting, as it would provide an algebraic foundation of the $F_K$ and $P_K$ series, and may lead to their full definition for knots and links. This seems more challenging. By constructing the ring $\Omega$ we make a small step in this direction.

\begin{theorem}
    The set $\Omega$ of inverted Habiro series that satisfy lower bound condition forms a $\Z\lr{q}$-algebra.
\end{theorem}


We can illustrate the relation of $J_K$, $P_K$ and $F_K$, and corresponding rings, in the following diagram:

\begin{center}
\begin{tikzcd}[row sep=tiny,column sep=10]
    \Lambda \arrow[r] &\Z[[y,q-1]] &\arrow[l]
    \Z[y,\frac{1}{\nabla_K(y)}][[q-1]] \arrow[r]
    & \Z[[y^{-1},q-1]] & \Omega  \arrow[l] \arrow[r] & \Z\lr{q}[[x]]\\
    J_K \arrow[r, mapsto] & R_{K}^{y \to 0} &\arrow[l, mapsto]
    R_{K} \arrow[r, mapsto]
    & R_{K}^{y \to \infty} & P_K \arrow[l, mapsto] \arrow[r, mapsto] & F_K
\end{tikzcd}
\end{center}

\subsection{Volume and perturbative series for the GM series}

In \cite{Kashaev_1994,Kashaev1995}, Kashaev defined a knot invariant, today known as Kashaev invariant. He formulated his famous Volume conjecture: Kashaev invariant of a knot $K$ grows exponentially, with the exponential factor being suitably normalized hyperbolic volume $\vol(K)$ of the knot complement \cite{Kashaev1997}. Kashaev invariant was later proved to coincide with the sequence  $J_n(\zeta_n)$, where $J_n$ denotes the $n$-th colored Jones polynomial and $\zeta_n = e^{2 \pi \iu/n}$ \cite{Murakami2001}.  

The conjecture was refined to include the subexponential terms of the asymptotic expansion, in the form of a formal (perturbative) series \cite{Gukov2005}. For the $4_1$ knot, Kashaev invariant has the following asymptotic expansion for large $n$ \cite{Garoufalidis2013}:
\begin{equation}\label{eq:vol_kashaev}
    J_n(\zeta_n) \sim e^{\frac{n\vol(4_1)}{2\pi} } n^{\frac{3}{2}} \Phi^J\left(\frac{2 \pi \iu}{n}\right),
\end{equation}
with
\begin{equation}\label{eq:kashaev_expn}
    \Phi^J(h) = \frac{1}{\sqrt[4]{3}} \left( 1 + 11\frac{h}{72 \sqrt{-3}}  + 697\frac{h^2}{2! (72 \sqrt{-3})^2} + \frac{724351}{5}\frac{h^3}{3! (72 \sqrt{-3})^3} \dots \right).
\end{equation}

It is natural to expect similar phenomena for the $F_K$ invariant. 
We demonstrate it on the $4_1$ knot, for which we can compute many coefficients $f_n$ using the quantum A-polynomial recurrence \cite[eqn. 173]{GM19} or the explicit form found by Park (see \eqref{eq:coef_FK_41}). The sequence $f_n(\zeta_{2n})$ has an asymptotic expansion of the form
\begin{equation}
    f_n(\zeta_{2n}) \sim e^{\frac{n \vol(4_1)}{\pi}} \Phi^F\left(\frac{2 \pi \iu}{n}\right),
\end{equation}
where
\begin{equation}
    \Phi^J(h) = \frac{1}{\sqrt{3}} \left( 1 + 4\frac{h}{72 \sqrt{-3}}  + 304\frac{h^2}{2! (72 \sqrt{-3})^2} + \frac{290912}{5}\frac{h^3}{3! (72 \sqrt{-3})^3} \dots \right)
\end{equation}
The first few coefficients of the above series read
\[
1,4,304,\frac{290912}{5},\frac{107155712}{5},\frac{91298182144}{7},\frac{416634955237376}{35},\frac{76199853915803648}{5}.
\]

They have similar denominators as the coefficients of the series $\Phi^J(h)$. P. Scholze noticed that the coefficients of $\log\Phi^J(h)$ have smaller denominators than those of $\Phi^J(h)$ \cite{GaroufalidisZagier2024}. The same appears to be true for the series $\log \Phi^F(h)$. Moreover, the denominators of the coefficients of $\log \Phi^J(h)$ and $\log \Phi^F(h)$ coincide at least up to degree $h^{30}$. In fact, we can refine this statement: If we consider the quotient $\Phi^J(h)/\sqrt{\Phi^{F}(h)}$ (this combination cancels the volume exponential), we obtain
\begin{equation}
\frac{\Phi^J(h)}{\sqrt{\Phi^{F}(h)}}=  1 + 9\frac{h}{72 \sqrt{-3}}  + 513\frac{h^2}{2! (72 \sqrt{-3})^2} + 109593\frac{h^3}{3! (72 \sqrt{-3})^3} \dots 
\end{equation}
Miraculously, the denominators cancel out and the coefficients are integers. This suggests a deeper relation between the perturbative expansions. The next coefficients are listed in \cref{tab:coefs}. 
The coefficient sequence appears to be periodic modulo any small integer, e.g., it is constant 1 modulo 8.
\begin{table}
\begin{tabular}{L|L}
0 & 1\\
1 & 9\\
2 & 513\\
3 & 109593\\
4 & 42906753\\
5 & 27317473641\\
6 & 25828136112897\\
7 & 33949071497981817\\
8 & 59208582624818022657\\
9 & 132319684039443707344329\\
10 & 368766761021357740013998593 \\
11 & 1254010413258251625144312947673 \\
13 & 5110963418460346801290854339484033 \\
14 &24595884356644903525696691616798665001 \\
15 &137997481943499655286659777079776274537217 \\
16 &892888649092011448566605775439461494729114937 \\
17 &6599841709252838201741098263399988038583452402177 \\
18 &55268971214377029297573735672254225188954729385535369\\
19 &520550298245098823705189091219704438511316884101446748673 \\
20 &5478356280328258810596895181463719662068445661595261680275353 \\
\end{tabular}
\caption{Integral coefficients of the series $\Phi^J(h)/\sqrt{\Phi^{F}(h)}$.}
\label{tab:coefs}
\end{table}

What about the sequence $f_n(\zeta_n)$? Somewhat surprisingly, it is periodic:
\[
    1, 2, 2, 1, \frac{3-\sqrt{5}}{2}, 1, 2, 2, 1, \frac{3-\sqrt{5}}{2} \dots.
\]
It would be interesting to explain these phenomena from the viewpoint of complex Chern--Simons theory and resurgence.

\subsection{Surgery formulas}

The series $F_K$ was invented as a tool to study the $q$-series invariant $\Zhat_a(Y,q)$ of a closed 3-manifold $Y$ with a $\spc$ structure $a$. Gukov and Manolescu proposed a surgery formula (GM formula), which converges for certain range of surgery coefficients. Park found a different, `regularized' formula in terms of the inverted Habiro series, and proposed that it should be used when the GM formula does not converge \cite[Conj. 4]{Park21}.

We interpret the GM formula as a residue computation. Using residue theorem, we recover Park's formula and make it more explicit. The resulting formula (see \cref{prop:p-surgery}) may be viewed as a unification of both formulas.

\subsection{Plan of the paper}
In Section 2, we give necessary background on $q$-analogs and their degrees. In Section 3, we recall inverted Habiro series, prove the multiplication formula and define the ring $\Omega$. In Section 4, we study the relation of $P_K(x,q)$ and $F_K(x,q)$ and prove \cref{thm:theta_intro}. In Section 5, we explore the residues for the simplest examples, $3_1^l$, $3_1^r$, $4_1$, and find that they are related to classical $q$-series identities. In Section 6, we study various ramifications of the $F_K$ invariant.
Finally in the last section, we study how the residue analysis can help to understand the surgery formulas for the $\Zhat$ invariant of closed 3-manifolds.

\medskip
\textbf{Acknowledgement.}
The author is grateful to Sergei Gukov, Shimal Harichurn, Mrunmay Jagadale, Sunghyuk Park and Lara San Mart\'in Su{\'a}rez for useful discussions. Most computations were performed using the Mathematica package QAlg written by  Davide Passaro, L. S. M. Su{\'a}rez and the author. The author was supported by the Simons Foundation Collaboration grant New Structures in Low-Dimensional Topology. 

\section{Preliminaries}

\subsection{\texorpdfstring{$q$}{q}-analogs}

Let $v$ be an indeterminate and set $q = v^2$. Consider the field $\C(v)$ of rational functions in $v$ over the complex numbers. For $n \in \Z$ and $k \in \Z^{\geq 0}$, we define the following elements of $\C(v)$:
\[
	[n] = \frac{v^n-v^{-n}}{v-v^{-1}},
	\quad
	[k]! = [k][k-1] \cdots [1],
	\quad
	\qbinom{n}{k} = \frac{[n][n-1]\cdots[n-k+1]}{[k]!}.
\]

We also define 
\[
	\sm{n} = v^n-v^{-n}, \quad \sm{n}!=\sm{k} \cdots \sm{1}=(v-v^{-1})^k [k]!, \quad \sm{n}_k=\sm{n} \cdots \sm{n-k+1}.
\]

Finally, we use the $q$-Pochhammer symbol notation. For $n \in \Z^{\geq 0} \cup \{\infty\}$ and $k \in \Z^{\geq 0}$, we define 
\[
(a;q)_n = \prod_{j=0}^{n-1} (1-aq^j), \quad (a_1,\dots,a_k;q)_n=(a_1;q)_n \cdots (a_k;q)_n,
\]
and use the usual shortcut $(q)_n=(q;q)_n.$ These are related to the expressions above, e.g., $\sm{n}_k = (-1)^k q^{k (k - 1 - 2 n)/4}  (q^{n- k + 1};q)_k.$

As a convention, we often omit the $q$-dependence in formulas, e.g., we write $F_K(x)$ instead of $F_K(x,q)$ and $f_k$ instead of $f_k(q)$.

\subsection{Valuation}

Let $L\lr{q}$ be the field of formal Laurent series over a field $L$ (usually $L=\Q$ or $L=\Q(x)$). For a series 
\[
	f(q)=\sum_{i=m}^\infty c_iq^i\in L\lr{q}, \quad c_m \neq 0,
\]
we write $\delta(f(q)) =m \in \Z \cup \{\infty\}$ for the \emph{degree} of $f$, i.e., the minimal exponent of $q$ in $f$. The degree $\delta$ is a valuation on $L\lr{q}$:
\begin{align*}
	\delta(f)   & = \infty \text{ if and only if } f=0,    \\
	\delta(fg)  & = \delta(f) + \delta(g),      \\
	\delta(f+g) & \geq \min(\delta(f),\delta(g)). 
\end{align*}
    
We analogously define $\delta_v$ on $L\lr{v}$, such that $\delta_v(f) = 2\delta(f)$ for $f \in L\lr{q}$ and collect $\delta_v$ of common polynomials. For $n \in \Z$ and $k \in \Z^{\geq 0}$, we have
\begin{equation}
	\delta_v(\sm{n}) =
	\begin{cases}
		-|n|   & \text{ if } n \neq 0 \\
		\infty & \text{ if } n = 0    
	\end{cases},
	\quad \delta_v([n]) = 
	\begin{cases}
		-|n-1| & \text{ if } n \neq 0 \\
		\infty & \text{ if } n = 0,   
	\end{cases}
\end{equation}

\begin{equation}\label{eq:deg_fact}
	\delta_v(\sm{n}_k) =
	\begin{cases}
		-nk + \binom{k}{2} & \text{ if } n > 0,                    \\
		nk - \binom{k}{2}  & \text{ if } n < 0,                    \\
		\infty                 & \text{ if }  n = 0 \text{ and } k >0, \\
		0                      & \text{ if } n = 0 \text{ and } k =0.  
	\end{cases}
\end{equation}

Using $\qbinom{n}{k}=\sm{n}_k/\sm{k}_k$, we get
\begin{equation}\label{eq:deg_binom}
	\delta_v\left( \qbinom{n}{k} \right) =
	\begin{cases}
		k(k-n) & \text{ if } n > 0,                    \\
		k(n+1) & \text{ if } n < 0,                    \\
		\infty & \text{ if }  n = 0 \text{ and } k >0, \\
		0      & \text{ if } n = 0 \text{ and } k =0.  
	\end{cases}
\end{equation}

\section{Inverted Habiro series}

In this section, we recall Park's inverted Habiro series of knots. We study, under which conditions it gives a $q$-series. We prove a multiplication formula for inverted Habiro series and check that the growth condition is preserved under multiplication. This gives a commutative ring $\Omega$ of inverted Habiro series, with a natural injective homomorphism to a ring of formal power series.

\subsection{Inverted Habiro series}

Let $x$ be an indeterminate. We consider a sequence $(\sigma_n)_{n \in \Z}$ of rational functions in $\Q(x+x^{-1},q)$  defined by
\begin{equation}\label{eq:sigmas}
	\sigma_{n} = 
	\begin{cases}
		\prod_{i=1}^n (x+x^{-1}-q^i-q^{-i})           & \text{ if } n \geq 0, \\
		\prod_{i=0}^{-n-1} (x+x^{-1}-q^i-q^{-i})^{-1} & \text{ if } n<0.      
	\end{cases}  
\end{equation}

It satisfies the following relation for all $n \in \Z$:
\begin{equation}\label{eq:recur_sigma}
	\sigma_n = (x+x^{-1}-q^n-q^{-n})\sigma_{n-1}.
\end{equation} 

For $n<0$, we set $k=-n-1 \geq 0$, and frequently use the following normalization: 
\begin{equation}\label{eq:tsigma}
	\tsigma_n = (x^{-1}-1) \sigma_n = \frac{(-1)^{k}q^{\binom{k+1}{2}}}{(x;q)_{k+1} (qx^{-1};q)_{k}}.
\end{equation}

Given a sequence of coefficients $(a_n)_{n<0}$, for now treated as indeterminates, the associated  \emph{inverted Habiro series} $P(x)$ is defined as a formal sum of the form\footnote{Our convention differs from Park's by an $x^{1/2}$ factor.}
\begin{equation}\label{eq:ihs_def}
	P(x) = \sum_{n <0} a_n \tsigma_n = (x^{-1}-1) \sum_{n < 0} a_{n} \sigma_{n}.  
\end{equation}

We also use a `reduced' version of $P(x)$ without the factor $(x^{-1}-1)$, which we denote by the same letter as the corresponding sequence, $a=\sum_{n <0} a_n \sigma_n.$ Note that we have $P(x^{-1})=-x P(x).$ 

\subsection{Lower bound condition}

The degrees of the coefficients $f_k$ of the series $F_K = \sum_{k=0}^\infty {x^k f_k}$ typically tend to $-\infty$ as $k \to \infty$. For example, for the knot $4_1$, the degrees are $\delta(f_k(q)) = -\lfloor k^2/4\rfloor$. Hence the radius of convergence of $F_K(x)$ is zero for any fixed $q \neq 0$. 

The inverted Habiro series $P_K(x)$ tends to have better convergence properties. In particular, it often defines a $q$-series, i.e., it gives an element of $\Q(x)\lr{q}$ when expanded at $q=0$, and it even defines  a meromorphic function of $x$ and $q$ on some domain in $\C^2$.
For this to be possible, the degress of the inverted Habiro coefficients $a_{n}$ must satisfy some lower bound. The following condition is deviced for this purpose:

\begin{definition}\label{def:lbc}
	A sequence $(a_n)_{n<0}$  in $\Z\lr{q}$ satisfies the \emph{lower bound condition (LBC)}, if for all $n<0$, we have
	\[
		\delta(a_n) \geq -\frac{n(n+3)}{2} + C,
	\]
	with some constant $C$. We say that the corresponding reduced (resp. unreduced) inverted Habiro series 
    \[
    a = \sum_{n<0} a_n \sigma_n, \quad \text{(resp. } P(x) = \sum_{n<0} a_n \tsigma_n)
    \]
    satisfies the LBC if its sequence of coefficients $(a_n)_{n<0}$ does. 
\end{definition}

\begin{remark}
The inverted Habiro coefficients of right-handed trefoil satisfy exactly $\delta(a_n)=-n(n+3)/2$. The constant $C$ is needed at the beginning of the sequence: For left-handed trefoil, we have $\delta(a_n)=n(n+3)/2 \geq -n(n+3)/2 -2.$ For the connected sum of $k$ left-handed trefoils, we have $C=-2k$, so the constant can be arbitrarily small.
\end{remark}


Note that the lower bound condition for the inverted Habiro series is sufficient for $P_K(x)$ to be a $q$-series. To further establish the meromorphicity, one would have to study the growth of the coefficients of each polynomial $a_{-n}$. This should not be a big issue due to the $q$-holonomic properties of the sequence $a_{-n}$.

\subsection{Multiplication formula}

We give a formula for the product of two (reduced) inverted Habiro series. Then we prove that the product preserves the lower bound condition. We set $\alpha_{m,n}=\sm{m-n}\sm{m+n+2}$. For any $m,n \in \Z$, the following identity easily follows from the recurrence relation \eqref{eq:recur_sigma} of $(\sigma_n)_{n \in \Z}$:
	\begin{equation}\label{eq:commutation}
		\sigma_{m} \sigma_{n+1} = \sigma_{m+1} \sigma_n + \alpha_{m,n} \sigma_m \sigma_n.
	\end{equation}

For $n \in \Z$, let $\sigma_n^0 \in \Z[q^{\pm 1}][[x]]$ be the Taylor expansion of $\sigma_n$ at $x=0$. 
For $m,n \in \Z$ and $i \geq 0$, we denote 
\begin{equation}
\gamma_{m,n}^i = \sm{m}_i \sm{n}_i \qbinom{m+n+1}{i}.
\end{equation}
The following proposition extends \cite[Prop.\ 9.16]{Habiro2007integral} to negative integers\footnote{Thanks to L.S.M. Su\'arez for suggesting to `just plug negative integers' into Habiro's formula.}. 
\begin{proposition}\label{prop:mult}
	For any $m,n \in \Z$, we have
	\begin{equation}\label{eq:mult_sigma_neg}
		\sigma_m^0 \sigma_n^0 = \sum_{i=0}^\infty  \gamma_{m,n}^i \sigma_{m+n-i}^0.
	\end{equation}
\end{proposition}

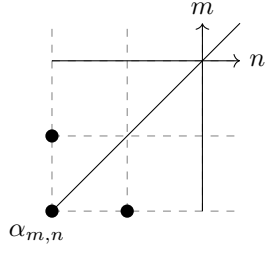
\begin{figure}
    \centering
    \begin{tikzpicture}[scale=1]
    \draw[help lines, dashed, step=1] (-2,-2) grid (0.5,0.5);

    \draw[->] (-2,0) -- (0.5,0) node[right] {\(n\)};
    \draw[->] (0,-2) -- (0,0.5) node[above] {\(m\)};

    \draw[] (-2,-2) --   (0.5,0.5);
    
    \foreach \x/\y in {-2/-2, -2/-1, -1/-2} {
        \fill[black] (\x,\y) circle (2.5pt);
    }
    \node at (-2.2,-2.3) {$\alpha_{m,n}$};
    \end{tikzpicture}
    \caption{A visualization of the terms $\sigma_m\sigma_n$ in \eqref{eq:commutation}.}
    \label{fig:enter-label}
\end{figure}

\begin{proof}
    First we take $n \geq 0$ and $m \in \Z$. We proceed by induction on $n$. The identity is trivial for $n=0$. 
	By expanding \eqref{eq:commutation} at $x=0$, we obtain
	\begin{align*}
		\sigma^0_m \sigma_{n+1}^0
		  & =\sigma_{m+1}^0 \sigma_n^0 +\alpha_{m,n} \sigma_m^0 \sigma_n^0 \\
		  & =
		\sum_{i \geq 0} \gamma_{m+1,n}^i \sigma^0_{m+n+1-i} + \alpha_{m,n} \sum_{i \geq 0} \gamma_{m,n}^i \sigma^0_{m+n-i}\\
		  & =\sum_{i \geq 0} (\gamma_{m+1,n}^{i} + \alpha_{m,n} \gamma_{m,n}^{i-1} ) \sigma^0_{m+n+1-i}. 
	\end{align*}
	Here we use a convention that $\gamma_{m,n}^i=0$ for $i<0$. The statement reduces to
	\begin{equation*}\label{eq:gammas}
		\gamma_{m,n+1}^{i} = \gamma_{m+1,n}^{i}  + \alpha_{m,n} \gamma_{m,n}^{i-1},
	\end{equation*}
    which reduces to an elementary identity
    \[
	\sm{m-i+1}\sm{n+1} = \sm{m+1}\sm{n-i+1}+\sm{m-n}\sm{i}.
    \]
    For $m,n<0$ with $m \neq n$, we can divide \eqref{eq:commutation} by $\alpha_{m,n}$ and proceed similarly by induction, assuming that we have proved the statement for all pairs with $n=m<0$. 
    
    For $n=m<0$, we have $\alpha_{n,n}=0$ and hence we cannot deduce the formula for $\sigma_n^2$ from other products. To prove the statement in this case, we expand both sides of \eqref{eq:mult_sigma_neg} in $x$ and compare the coefficients of $x^l$ for $l=0,1,2, \dots$ This leads to the following polynomial identity ($k=-n-1 \geq 0$):
\begin{multline}
    \sum_{i=0}^l \qbinom{2k+i}{i}\qbinom{2k+l-i}{l-i}
    =\\ -\sum_{i=0}^l \sm{k+i}^2_i\qbinom{2k+i}{i}
\qbinom{4k+2+i+l}{l-i}\left( 1-\frac{\sm{l-i}}{\sm{4k+2+i+l}} \right)
\end{multline}
We can prove it using the method of $q$-holonomic functions. Namely we check that both sides satisfy the same $q$-recurrence and they agree for couple of initial values\footnote{A Mathematica notebook \cite{Mathematica} with a proof using the HolonomicFunctions package \cite{Koutschan10} is available at \url{https://github.com/jxs2628/inverted_habiro}. In principle, one could use this method for all pairs $m,n$. However, the computer computation with an additional variable is more involved and I was not able to run it.}.
\end{proof}

\subsection{The ring of inverted Habiro series}

We define a ring $\Omega$ consisting of (reduced) inverted Habiro series satisfying the lower bound condition. The multiplication is given by the formula in \cref{prop:mult}. For the purpose of this section, $\sigma_n$ are merely indeterminates rather than rational functions.  Although we usually consider series in $ \sum_{n<0} a_n \sigma_n$, here we allow for finitely many positive indices as well\footnote{This allows to take the constant function $\sigma_0 = 1$ as the inverted Habiro series of the unknot.}.

\begin{definition}
    Let $\Omega$ be the collection of formal sums $a=\sum_{m \leq M} a_m \sigma_m$ with $M \in \Z$, such that $(a_m)_{m<0}$ satisfies the lower bound condition. Using \cref{prop:mult}, we define the product of two elements $a =\sum_{m \leq M} a_m \sigma_m, b=\sum_{n \leq N} b_n \sigma_n \in \Omega$ as
    \begin{equation}\label{eq:prod_formula}
    	\begin{split}
    		ab &= \sum_{m \leq M,n \leq N} a_m b_n \sum_{i \geq 0} \gamma_{m,n}^{i} \sigma_{m+n-i}\\
    		&=
    		\sum_{l \leq 0} \bigg( \sum_{\substack{m \leq M, n \leq N\\ m+n \geq l}}
    		\gamma_{m,n}^{m+n-l} a_m b_n \bigg) \sigma_l.
    	\end{split}
    \end{equation}
\end{definition}

The following proposition shows that $\Omega$ is a well-defined $\Z\lr{q}$-algebra with a natural injective homomorphism to a ring of power series in $q$ and $x$:
\begin{proposition}\label{prop:subalg}
	The map given by $\sum_{n \leq N} a_n \sigma_n \mapsto \sum_{n \leq N} a_n \sigma_n^0$ realizes $\Omega$ as a $\Z\lr{q}$-subalgebra of $\Z\lr{q}[[x]]$.
\end{proposition}

\begin{proof}
	The sum $\sum_{n \leq N} a_n \sigma_n^0$ gives a well-defined element of $\Z\lr{q}[[x]]$ as $\sigma_n^0 \in x^{-n}\Z[q^{\pm 1}][[x]]$ for $n<0$, so that the coefficient of each power of $x$ is a finite expression in the coefficients $a_n$. 
    
    Clearly $\Omega$ is a $\Z\lr{q}$-submodule of $\Z\lr{q}[[x]]$. Consider $a,b \in \Omega$ and their product $c = \sum_{l \leq M+N } c_l \sigma_l$ as above, with 
    \begin{equation}\label{eq:prod_coef}
		c_l = \sum_{\substack{m \leq M,n \leq N\\ m+n \geq l}}
		\gamma_{m,n}^{m+n-l}  a_m b_n =  \sum_{\substack{m \leq M,n \leq N\\ m+n \geq l}}
		 \sm{m}_{m+n-l}\sm{n}_{m+n-l}\qbinom{m+n+1}{m+n-l}  a_m b_n. 
	\end{equation}
	We need to bound the degree $\delta_v(c_l)$.
    For $m,n<0$ and $l \leq m+n$, \eqref{eq:deg_fact}, \eqref{eq:deg_binom} give:
\begin{equation}
	\delta_v(\gamma_{m,n}^{m+n-l})= (m+n)(m+n+3)-l(l+3).
\end{equation}
    Then we have:
	\begin{equation*}
		\begin{split}
			\delta_v(\gamma_{m,n}^{m+n-l}  a_m b_n) 
			&= (m+n)(m+n+3)-l(l+3) + \delta_v(a_m)+\delta_v(b_n) + C_a + C_b \\
			&\geq (m+n)(m+n+3)-l(l+3) -m(m+3) -n(n+3) + C\\
			&= -l(l+3) + 2mn + C\\
			&\geq -l(l+3) + C.
		\end{split}
	\end{equation*}
    The triangle inequality gives
	\begin{equation*}
			\delta_v(c_l) \geq 
			\min_{\substack{m \geq M,n \leq N \\ m+n \geq l}} \gamma_{m,n}^{m+n-l} a_m b_n \geq -l(l+3) + C.\\
	\end{equation*}	
    
    Consider now the case $M \geq m > 0$ and $n<0$. By taking a sufficiently large constant $C$ for the sequence $c_l$ we may only consider $l$ such that $l<-M-2$. Moreover,
    $\sm{m}_{m+n-l}=0$ unless $n \leq l$.
    Combining those two conditions gives $n+m+1 \leq l+m+1 <-M + m -1< 0$. In this case, using \eqref{eq:deg_fact} and \eqref{eq:deg_binom}, we obtain
    \begin{equation*}
		\begin{split}
			&\delta_v(\gamma_{m,n}^{m+n-l} a_m b_n) =\\ 
			&= (n-m)(m+n-l) + (m+n-l)(n+m+2) + \delta_v(a_m) + \delta_v(b_n)+C_a+C_b \\
            &\geq (m+n-l)(2n+2)-m(m+3)-n(n+3)+C\\
            &= -l(l+3) + \left(n-l-\frac12\right)^2-2m(1+n) -m(m+3) + C-\frac14\\
            &\geq -l(l+3)+C'.
        \end{split}
	\end{equation*}
    The claim again follows from the triangle inequality.
    The other cases ($m=0$ or $m,n>0$) are straightforward and we leave them to the reader.
\end{proof}

\begin{remark}
    We have two subrings of $\Omega^{<0} \subset \Omega^{\leq 0} \subset \Omega$ given by the series  $\sum_{n<0} a_n \sigma_{n}$ or $\sum_{n \leq 0} a_n \sigma_{n}$ respectively ($\Omega^{<0}$ is non-unital).
\end{remark}

\begin{remark}
	It is expected that the multiplication of (reduced) inverted Habiro series corresponds to taking the connected sum of knots. We can formulate \cref{prop:subalg} as the statement that the set of knots for which the inverted Habiro series satisfies the LBC is closed under taking the connected sum.
\end{remark}

\section{Inverted Habiro series and the GM series}

\subsection{Expansion at zero}

We collect basic facts about the term-by-term expansion of an inverted Habiro series $P(x)$ at $x=0$.  The series $P(x)$ has poles at $x=q^j$ for each $j$, so for any given $|q|<1$, the poles accumulate at $x=0$. We cannot speak of an expansion at $x=0$ of $P(x)$ as a meromorphic function, but we can consider a `term-by-term' expansion, where we Taylor expand each $\tsigma_{-k-1}$ and sum over $k$. 

For each function $\tsigma_{-k-1}$ we consider its Taylor expansion $\tsigma^0_{-k-1}$ at $x=0$: 
\begin{equation}
		\tsigma^0_{-k-1} =
		x^{k}  \sum_{j \geq 0} x^{j} \qbinom{2k+j}{j} \in x^k \Z\lr{q}[[x]].
\end{equation}

By the $x$-expansion of $P(x) =  \sum_{k \geq 0} a_{-k-1} \tsigma_{-k-1}$, we mean the expansion at $x=0$ of $P(x)$:
\begin{equation}\label{eq:coefs_fk}
    x^{-1/2}F(x) = \sum_{i \geq 0} f_i x^i \colonequals \sum_{k \geq 0} a_{-k-1} \tsigma^0_{-k-1}.
\end{equation}
We recall the following:
\begin{proposition}[{\cite[Prop.\ 2]{Park21}}]\label{prop:ih_fk}
	The sequence $(f_i)_{i \geq 0}$ is uniquely determined by the sequence $(a_{-k-1})_{k \geq 0}$ and vice versa via the following formulas:
	\begin{equation}\label{eq:tmFH}
		f_i = \sum_{k=0}^i \qbinom{k+i}{2k} a_{-k-1}, \ \quad \quad    
		a_{-k-1} = \sum_{i=0}^k (-1)^{k+i} \qbinom{2k}{k-i} \frac{[2i+1]}{[k+i+1]} f_i.
	\end{equation}
\end{proposition}

\begin{remark}
    This relation is similar to the relation of the colored Jones polynomials $V_n$ and usual Habiro polynomials $a_n$ (see \cite[Lemma 6.1.]{Habiro2007}). A notable difference is that in Habiro's work, the fact that $a_n \in \Z[q^{\pm 1}]$ is highly nontrivial, while here the fact that $a_{-k-1} \in \Z[q^{\pm 1}]$ (or $\Z\lr{q}$) if and only if the $f_i \in \Z[q^{\pm 1}]$ (or $\Z\lr{q}$) is clear, as the coefficients in \eqref{eq:tmFH} are Laurent polynomials.
\end{remark}

    Assuming the LBC for inverted Habiro coefficients $a_{-n}$, \cref{prop:ih_fk} implies a similar growth condition for the coefficients $f_i$: 
	\begin{lemma}\label{lem:LBC_FK}
		For an LBC sequence $(a_n)_{n<0}$ (with some constant $C$), the coefficients $f_i$ satisfy
		\begin{equation}
			\delta(f_i) \geq -\binom{i}{2} +C +2.
		\end{equation}
	\end{lemma}
	
	\begin{proof}
		The triangle inequality gives
		\begin{equation*}
			\delta_v(f_i) = \delta_v\left( \sum_{k=0}^i \qbinom{k+i}{2k} a_{-k-1} \right) \geq \delta_v\left( \qbinom{k+i}{2k} a_{-k-1}\right)
		\end{equation*}
		for some $k$ with $0 \leq k \leq i$.
		For every such $k$ we have
		\begin{align*}
			\delta_v\left( \qbinom{k+i}{2k} a_{-k-1}\right) &
			\geq 2k(2k-k-i) - (-k-1)(-k+2) + C \\
			& = k^2-2ki+k+2+C \\
			& \geq -i(i-1) + C+2, 
		\end{align*}
as the minimum is attained for $k=i$.
	\end{proof}

\cref{prop:ih_fk} shows that the sequence $(f_k)_{k \geq 0}$ of coefficients $x^{-\frac12}F(x)$ determines the sequence $(a_{-k-1})_{k \geq 0}$ of coefficients of $P(x)$ and vice versa. In fact, the inverted Habiro series has an elegant form in terms of the coefficients $f_k$:

\begin{theorem}\label{thm:theta_identity}
	Assuming the lower bound condition, we have
	\begin{equation}\label{eq:ih_prod}
		P(x)=\frac{1}{(q,x,qx^{-1};q)_\infty } \sum_{i \geq 0} f_i \theta_i(x),
	\end{equation}
    where
    \begin{equation}
        \theta_i(x) = (-1)^iq^{\binom{i+1}{2}} \left(1 + \sum_{n=1}^\infty (-1)^nq^{\binom{n+1}{2}+ni} (x^n+x^{-n})\right).
    \end{equation}
\end{theorem}

\begin{proof}
    By the lower bound condition and \cref{lem:LBC_FK}, 
    \[
        \delta(f_i \theta_i(x)) \geq -\binom{i}{2} + \binom{i+1}{2} +C = i + C.
    \] 
    Therefore both sides of \eqref{eq:ih_prod} converge formally as a $q$-series.
    
    We express each $a_{-k-1}$ in terms of the coefficients $f_i$ using \eqref{eq:tmFH}, to obtain:
    \begin{equation}
    	P(x)=\sum_{k \geq 0} a_{-k-1} \tsigma_{-k-1} 
    	= \sum_{i \geq 0} f_i \sum_{k \geq i} (-1)^{k+i} \qbinom{2k}{k-i} \frac{[2i+1]}{[k+i+1]} \tsigma_{-k-1}.
    \end{equation}
	For any $i \geq 0$, the following $q$-series identity holds:
	\begin{equation}\label{eq:theta}
        \sum_{k \geq i}  (-k)^{k+i} \qbinom{2k}{k-i} \frac{[2i+1]}{[k+i+1]} \tsigma_{-k-1} =  
        \frac{\theta_i(x)}{(q,x,qx^{-1};q)_\infty} 
	\end{equation}
    To show this, we shift the summation index by $k \to k+i$ on the left-hand side:
	\begin{equation*}
		\begin{split}
		&\sum_{k \geq 0} (-1)^k \qbinom{2k+2i}{k} \frac{[2i+1]}{[k+2i+1]} \frac{(-1)^{k+i}q^{\binom{k+i+1}{2}}}{(x,qx^{-1};q)_{k+i}}\\
		&=\sum_{k \geq 0} (-1)^i q^{\binom{i+1}{2}} \frac{(q)_{2k+2i} (1-q^{2i+1})q^k}{(q)_k(q)_{k+2i+1} (x,qx^{-1};q)_{k+i}}\\
		&=\sum_{k \geq 0} (-1)^i q^{\binom{i+1}{2}} \frac{(q)_{2i} (q^{2i+1})_{2k} (1-q^{2i+1})}{(q)_{2i+1}(q)_k (q^{2i+2})_k (q^{i+1}x,q^{i+1}x^{-1};q)_k (x,qx^{-1};q)_{i }} \\
		&=\sum_{k \geq 0} \frac{(-1)^i q^{\binom{i +1}{2}}}{(qx,qx^{-1};q)_\infty} \frac{ (q^{2i +1})_{2k} (q^{i +1}x,q^{i +1}x^{-1};q)_\infty}{ (1-x) (q,q^{i +1}x,q^{i +1}x^{-1},q^{2i +2};q)_k}.
		\end{split}
	\end{equation*}
	The formula \eqref{eq:theta} now follows by setting $a=q^{i+1}x$, $b=q^{i+1} x^{-1}$ in Warnaar's identity
		\begin{equation}\label{eq:warnaar}
			1 + \sum_{n=1}^\infty (-1)^nq^{\binom{n}{2}} (a^n + b^n) = (q,a,b;q)_\infty \sum_{n=0}^\infty \frac{(ab/q;q)_{2n}q^n}{(q,a,b,ab;q)_n}.
		\end{equation}
\end{proof}

We observe that the inverted Habiro series `regularizes' the formal power series $x^{-1/2}F(x) = \sum_{i  \geq 0} f_i  x^i $. By that we mean that in \eqref{eq:ih_prod} each $f_i $ gets multiplied by a term whose $q$-degree is at least $\binom{i +1}{2}$, cancelling the negative powers in $f_i $ up to a constant.
This can be viewed as a possible explanation of why the inverted Habiro series leads to `regularized' formulas for the $\Zhat(q)$ invariant (see \cref{sec:surgery}).

\section{Residues}

In this section, we study the residues of the inverted Habiro series. 

\subsection{Rational functions}

We start with computing the residues of $\tsigma_{n}$ for $n<0$. We set $k=-n-1 \geq 0$ so that $n=-k-1$.

\begin{lemma}\label{lem:res_hab}
	For $n<0$, the form $\tsigma_{n}(x)dx$ is meromorphic on $\C\PP^1 \times \Delta$ with simple poles along $x=q^j$ for $j = 0, \pm 1,\pm 2,\dots, \pm k$. Additionally, $\sigma_{-1}dx$ has a simple pole along $x=\infty$. 
    The corresponding residues are:
	\begin{equation}\label{eq:res}
		\res_{x=q^j} \tsigma_{-k-1}dx =
        -\frac{(-1)^{k+j} q^{ \binom{j+1}{2}+ \binom{k+1}{2}}}{(q)_{k-j} (q)_{k+j}}
    \end{equation}
    and
	\begin{equation}\label{eq:res_infty}
		\res_{x=\infty} \tsigma_{-1}dx = 1.
	\end{equation}
\end{lemma}

\begin{proof}
	We compute the residues. For $j > 0$,
	\begin{equation*}
		\begin{split}
			\res_{x=q^j} \tsigma_{-k-1}dx 
            &= \lim_{x \to q^j} x(1-q^jx^{-1}) \frac{(-1)^k q^{\binom{k+1}{2}}}{(x;q)_{k+1}(qx^{-1};q)_k}\\
			&=\frac{(-1)^k q^{\binom{k+1}{2}+j}}{(q^j;q)_{k+1}(q^{-j+1};q)_{j-1}(q)_{k-j}} =\frac{(-1)^{j-1+k} q^{\binom{k+1}{2}+j+\binom{j}{2}}}{(q^j;q)_{k+1}(q)_{j-1}(q)_{k-j}},
		\end{split}
	\end{equation*}
    and the statement follows using $(q)_{j-1} (q^j;q)_{k+1}= (q)_{k+j}$. The case of $j \leq 0$ is analogous.
\end{proof}
\begin{remark}
	The residues satisfy 
    \begin{equation}
    \res_{x=q^j} \tsigma_{n} dx = q^{j} \res_{x=q^{-j}}\tsigma_{n}dx.
	\end{equation}
    This symmetry is obvious from \eqref{eq:res} or by noting that $\tsigma_n(x) dx=x \tsigma_n(x^{-1})d(x^{-1})$.
\end{remark}

Any rational function can be expressed using its residues and one checks easily the following
\begin{lemma}\label{lem:sigma_pole_expn}
    For any $k \geq 0$, we have
    \[
    \tsigma_{-k-1} = (1-x^{-1})\sum_{j=-k}^{k} \frac{\res_{x=q^j} \tsigma_{-k-1} dx}{x+x^{-1}-q^j-q^{-j}}.
    \]
\end{lemma}

\begin{remark}
	One should compare the formulas \eqref{eq:res} with the values of $\sigma_n$ for $n>0$:
    \[
		\sigma_n(q^j)=\prod_{i=1}^n \sm{j+i}\sm{j-i}.
	\]
    These polynomials form the transition matrix between the Habiro coefficients $a_n$ and colored Jones polynomials $J_j$.
\end{remark}

\subsection{Residues of the inverted Habiro series}

For an inverted Habiro series $P(x) = \sum_{n <0} a_{n} \tsigma_{n}$, we define its residues term-by-term, as we did with the $x$-expansion. We denote the residue at $x=q^j$ as $r_j$ and the residue at $x=\infty$ as $r_\infty(=a_{-1})$. As there is no residue at $x=0$, this notation should not lead to confusion. The lower bound condition ensures $r_j 
\in \Z\lr{q}$. We have
\begin{equation}\label{eq:res_series}
	r_j \colonequals \sum_{n < 0} a_{n} \res_{x=q^j} \tsigma_{n} dx = - \sum_{k \geq |j|} a_{-k-1}\frac{(-1)^{k+j} q^{\binom{k+1}{2}+\binom{j+1}{2}} }{(q)_{k+j}(q)_{k-j}}.
\end{equation}
In particular, for $j=0$ we have
\begin{equation}\label{eq:res_series_zero}
	r_0 = - \sum_{k=0}^\infty a_{-k-1} \frac{(-1)^kq^{\binom{k+1}{2}}}{(q)_k^2}.
\end{equation}
From \cref{lem:sigma_pole_expn}, we obtain the pole expansion of the inverted Habiro series:
\begin{proposition}
    \[
    P_K(x)
    = \sum_{j=-\infty}^\infty \frac{r_j}{x+x^{-1}-q^j-q^{-j}}.
    \]
\end{proposition}

\subsection{Residue theorem}
We need a residue theorem for the inverted Habiro series:
\begin{proposition}\label{prop:res_thm}
For an inverted Habiro series $P(x)$ which satisfies LBC, we have
	\begin{equation}
		\sum_{j \in \Z} r_j = -r_\infty (= -a_{-1} = -f_0).
	\end{equation}
\end{proposition}

\begin{proof}
    We have 
    \[
    \sum_{k \geq 0}\left(a_{-k-1}\sum_{j \in \Z} \res_{x=q^j} \tsigma_{-k-1}\right) =
    a_{-1} \delta_{k,0},
    \]
    by the (usual) residue theorem for rational functions $\tsigma_{-k-1}$.
    To be able to interchange the sum above, it is sufficient to check that 
    $\delta(a_{-k-1} \res_{x=q^j} \tsigma_{-k-1}) \to \infty $
    as $\max(k,\abs{j}) \to \infty$ (see \cite[Prop.\ 4.1.3]{Gouva1993}). This follows from the LBC condition:
    \begin{equation*}
    \begin{split}
        \delta(a_{-k-1} \res_{x=q^j} \tsigma_{-k-1})
        &\geq \binom{j+1}{2} + \binom{k+1}{2} - \binom{k}{2} + C\\ &= \binom{j+1}{2} + k + C \to \infty \text{ as } \max(k,|j|) \to \infty.
    \end{split}
    \end{equation*}
\end{proof}

\subsection{Residues and the GM coefficients}
In this section, we provide formulas relating the coefficients $f_k$ of the expansion \eqref{eq:coefs_fk} of $P(x)$ at $x=0$, and the residues $r_j$ of $P(x)$.
\begin{lemma}\label{lem:fk_from_residues}
For any $k \geq 0$,
    \begin{equation}\label{eq:fk_from_residues}
    f_k = - r_0 -\sum_{j \geq 1} (q^{-j(k+1)}+q^{j k}) r_j.
    \end{equation}
Conversely, for any $j \in \Z$:
\begin{equation}\label{eq:residues_from_fk}
	r_j = -\frac{q^{\binom{j+1}{2}}}{(q)^3_{\infty}} \sum_{k = 0}^\infty f_k (-1)^{k+j}q^{\binom{k+1}{2}} 
    \left(1+ \sum_{n=1}^\infty (-1)^nq^{\binom{n+1}{2}+nk}(q^{nj}+q^{-nj}) \right). 
\end{equation}
\end{lemma}
\begin{proof}
For \eqref{eq:fk_from_residues}, we use the residue theorem (as in \cref{prop:res_thm}):
\begin{equation*}
\begin{split}
    -f_k &= -\res_{x=0} x^{-k-1} P(x) dx\\
    &= \res_{x=1} x^{-k-1} P(x) dx +\res_{x=\infty} x^{-k-1} P(x) dx + \sum_{j \geq 1}q^{j(-k-1)} r_j + \sum_{j \geq 1}q^{j(k+1)} r_{-j}\\
    & = r_0 +0 +\sum_{j \geq 1} (q^{-j(k+1)}+q^{j k}) r_j
\end{split}
\end{equation*}
The formula  \eqref{eq:residues_from_fk} follows by taking the residues in \eqref{eq:ih_prod}.
\end{proof}

\subsection{Examples}
We compute the residues of the inverted Habiro series for trefoils and the figure-eight knot and relate them to well-known $q$-series. 
\subsubsection{Left-handed trefoil} 
The (inverted) Habiro coefficients of left-handed trefoil $3_1^l$ are 
\[
a_{-k-1}=(-1)^{k+1}q^{(k+1)(k-2)/2}.
\]
From \eqref{eq:res_series}, we obtain
\begin{equation}
	r_j=\sum_{k \geq |j|} \frac{(-1)^{j} q^{k^2-1+\binom{j+1}{2}} }{(q)_{k+j}(q)_{k-j}}.
\end{equation}
In particular for $j=0$ we have
\begin{equation}
	r_0 =q^{-1} \sum_{k=0}^\infty \frac{q^{k^2}}{(q)_k^2}=  q^{-1} + 1 + 2q + 3q^2 + 5q^3 + 7q^4 + 11q^5 + 15q^6 + 22q^7 +\dots
\end{equation}
The coefficients $r_0$ are just integer partitions and in fact, all residues admit a simple formula:
\begin{equation}
	r_j=\frac{(-1)^jq^{\frac{j(3j+1)}{2}-1}}{(q)_\infty}.
\end{equation}
This follows from a well-known formula, see \cite[eq. 2.4]{Andrews2014}. In particular, the residues satisfy a simple $q$-recurrence relation 
\[
r_{j+1}=-q^{3j+2} r_j.
\]
The residue theorem (\cref{prop:res_thm}) applied to the inverted Habiro series $P(x)$ is just Euler's pentagonal number theorem: 
	\begin{equation*}
		\sum_{j \in \Z} r_j= (q)_\infty^{-1}\sum_{j \in \Z} (-1)^jq^{j(3j+1)/2-1} = q^{-1} = - r_{\infty}.
	\end{equation*}
The coefficients of the GM series are
\[
    f_k = -\legendre{3}{2k+1} q^{-k(k+1)/6-1},
\]
where $\legendre{\cdot}{\cdot}$ denotes the Jacobi symbol.

\subsubsection{Right-handed trefoil} 
The (inverted) Habiro coefficients of right-handed trefoil $3_1^r$ are 
\[
a_{-k-1}=(-1)^{k+1}q^{-(k+1)(k-2)/2}.
\]
and by \eqref{eq:tsigma}, \eqref{eq:ihs_def}, we have
\begin{equation}\label{eq:ihs_rh_trefoil}
	P_{3_1^r}(x)
    = -q \sum_{k \geq 0} \frac{q^k}{(x;q)_{k+1} (qx^{-1};q)_k}.
\end{equation}
From \eqref{eq:res_series}, we obtain
\begin{equation}\label{eq:wedge_char}
	r_j=\sum_{k \geq |j|} \frac{(-1)^{j} q^{k+1+\binom{j+1}{2}} }{(q)_{k+j}(q)_{k-j}}.
\end{equation}
In particular, for $j=0$, we have 
\begin{equation}
	r_0 = \sum_{n \geq 0} \frac{q^{n+1}}{(q)^2_n}.
\end{equation}
The first few residues are collected in \cref{tab:res_31R}. 
A different expression for the residues can be obtained as follows. 
By Ramanujan's formula, which can be deduced from \eqref{eq:warnaar}, $P_{3_1^r}(x)$ in  \eqref{eq:ihs_rh_trefoil} equals
\begin{equation}\label{eq:ram_rhs}
	-q \sum_{k = 0}^\infty (-1)^k x^{3k}q^{k(3k+1)/2}(1-x^2q^{2k+1}) -q \frac{\sum_{k = 0}^\infty (-1)^k x^{2k+1}q^{\binom{k+1}{2}}}{(x;q)_\infty (qx^{-1};q)_\infty}.
\end{equation}
The first term in \eqref{eq:ram_rhs} coincides with the expansion $x^{-\frac12}F_{3_1^r}$  of $P_{3_1^r}(x)$ at $x=0$. It is holomorphic in $x$, so it does not contribute to the residues. Note that the fact that $x^{-\frac12}F_{3_1^r}$ defines a holomorphic function is a special property of the right-handed trefoil, or, more generally, of positive knots.

The residues of the second term in \eqref{eq:ram_rhs} are given by:
\begin{equation}
\begin{split}\label{eq:res_rh_trefoil_theta}
	r_j &=\frac{(-1)^{j}q^{\binom{j+1}{2}+1}}{(q)^2_\infty}\sum_{k = 0}^\infty (-1)^kq^{j(2k+1)}q^{\binom{k+1}{2}} \\
    &=\frac{(q)^{-2}_\infty (-1)^jq^{-j(3j-1)/2+1}}{(q)^2_\infty}\sum_{k=2j}^\infty (-1)^kq^{\binom{k+1}{2}}.
\end{split}
\end{equation}
The residue theorem (\cref{prop:res_thm}) for $P(x)$ gives the classical identity of Hecke--Rogers \cite[p. 42, eq. 4.28]{Andrews1979}:
\begin{equation}
    \sum_{j = -\infty}^\infty \sum_{k=2j}^\infty (-1)^{k+j} q^{(-3j^2 +j + k^2 +k)/2} = (q)_\infty^2.
\end{equation}

The residues $r_j$ satisfy an inhomogeneous $q$-recurrence
\begin{equation*}
	r_{j+1} = -q^{-3j-1} r_j - \frac{(-1)^jq^{\binom{j+1}{2}} (q^{-2j}-q)}{ (q)_\infty^2}.
\end{equation*}
We note that this implies that the sequence of normalized series $(-1)^j q^{-\binom{j+2}{2}}r_j$ stabilizes to $(q)_\infty^{-2}$ as $j \to \infty$.

\begin{remark}
The $q$-series $r_j$ of right-handed trefoil in the form of \eqref{eq:wedge_char} were studied in \cite{Datta2016} as `wedge characters'. A combinatorial proof of the formula \eqref{eq:res_rh_trefoil_theta} was given there.
\end{remark}

\begin{table}
\centering
\begin{tabular}{ L L }
\hline\hline
j   & r_j\\ 
\hline\hline
0    & q +q^2 + 3q^3 + 6q^4 + 12q^5 + 21q^6 + 38q^7 + 63q^8 + 106q^9 + 170q^{10} +\dots\\
1    & -(q^3 + 2q^4 + 5q^5 + 9q^6 + 18q^7 + 31q^8 + 55q^9 + 91q^{10} + 151q^{11} + 240q^{12} +\dots)\\
2    & q^6 + 2q^7 + 5q^8 + 10q^9 + 20q^{10} + 35q^{11} + 63q^{12} + 105q^{13} + 175q^{14} + 280q^{15} +\dots\\
3    & -(q^{10} + 2q^{11} + 5q^{12} + 10q^{13} + 20q^{14} + 36q^{15} + 65q^{16} + 109q^{17} + 183q^{18} +295q^{19} +\dots)\\
4    & q^{15} + 2q^{16} + 5q^{17} + 10q^{18} + 20q^{19} + 36q^{20} + 65q^{21} + 110q^{22} + 185q^{23} + 299q^{24} +\dots\\
\hline\hline
\end{tabular}
\caption{Residues of right-handed trefoil.}
\label{tab:res_31R}
\end{table}

\subsubsection{Figure-eight knot $4_1$}  For the $4_1$ knot, the inverted Habiro coefficients are $a_{-k-1} = 1$ for all $k \geq 0$. By \eqref{eq:tmFH}, the coefficients $f_n$ are:
\begin{equation}\label{eq:coef_FK_41}
    f_n = \sum_{i=0}^n \qbinom{n+i}{2i}.
\end{equation}
From \eqref{eq:res_series} we obtain 
\begin{equation}\label{eq:res_41_1}
	-r_j= \sum_{k \geq |j|} \frac{(-1)^{k+j} q^{\binom{k+1}{2}+\binom{j+1}{2}} }{(q)_{k+j}(q)_{k-j}}.
\end{equation}

The first few residues are collected in \cref{tab:res_41}. The residues can be also written as
\begin{equation}\label{eq:res_41_2}
    -r_j = (q)_{\infty}\sum_{k \geq |j|} \frac{(-1)^{k+j}  q^{(3k^2 +k-j^2+j)/2}}{ (q)_{k-j}
   (q)_{k+j}(q)_k} = \sum_{m,n \geq 0} (-1)^{n+m} \frac{q^{\binom{m+n+1}{2}+2mj+n j+j(j+1)}}{(q)_m (q)_n }.
\end{equation} 

The proof of the identities in \eqref{eq:res_41_2} can be obtained by a minor modification of the argument in \cite[A.1]{GaroufalidisZagier2023}.
In particular for $j=0$, we recover the series first studied in \cite[p. 15]{Garoufalidis2011} in relation to the asymptotic of Kashaev invariant:
\begin{equation}
	-r_0 = \sum_{k=0}^\infty \frac{(-1)^kq^{\binom{k+1}{2}}}{(q)_k^2} =
	(q)_\infty \sum_{k=0}^\infty (-1)^k \frac{q^{(3k^2+k)/2}}{(q)^3_k}.
\end{equation}

The residue theorem (\cref{prop:res_thm}) gives the following $q$-series formula:
\begin{equation}
    \sum_{j = -\infty}^{\infty} \sum_{k \geq \abs{j}} \frac{(-1)^{k+j}  q^{(3k^2 +k-j^2+j)/2}}{ (q)_{k-j}
   (q)_{k+j}(q)_k} = \frac{1}{(q)_\infty}. 
\end{equation}

\begin{table}
\centering
\begin{tabular}{L L}
\hline\hline
j   & r_j\\ 
\hline\hline
0   & -1 +q + 2q^2 + 2q^3 + 2q^4 -q^6 - 5q^7 - 7q^8 - 11q^9 +\dots\\ 
1   &-q^2 -q^3 -q^4 +q^6 + 3q^7 + 5q^8 + 7q^9 + 9q^{10} + 10q^{11}+\dots\\
2   & -q^6 -q^7 - 2q^8 - 2q^9 - 3q^{10} - 2q^{11} - 2q^{12} + 2q^{14} + 6q^{15}+\dots\\ 
3   &-q^{12} -q^{13} - 2q^{14} - 3q^{15} - 4q^{16} - 5q^{17} - 7q^{18} - 7q^{19} - 8q^{20} -8q^{21}+\dots\\ 
4   & -q^{20} -q^{21} - 2q^{22} - 3q^{23} - 5q^{24} - 6q^{25} - 9q^{26} - 11q^{27} - 15q^{28} - 17q^{29}+\dots\\
\hline\hline
\end{tabular}
\caption{Residues of the $4_1$ knot.}
\label{tab:res_41}
\end{table}

\section{Related \texorpdfstring{$q$-series}{q-series}}

In this section, we look at some further $q$-series related to the GM series and inverted Habiro series. We mostly focus on the example of figure-eight knot $4_1$. The topics of this section will be explored in greater detail and generality in a subsequent paper \cite{OSSS25}.

\subsection{Tails}

As for the colored Jones polynomial, $F_K$ series seems to admit `tails'. In the simplest means, the normalized coefficients $q^{-\delta(f_k)}f_k$ stabilize to a $q$-series. For $4_1$ knot, we obtain two $q$-series, one for even indices and one for odd. They can be expressed as quotients of Ramanujan theta functions $\phi(q), \psi(q), f(q)$. Even terms $f_{2n}$ stabilize to:
\begin{equation*}
    \frac{\phi(q)}{f(-q)} =(q;q)_{\infty}^{-1} \sum_{n=-\infty}^{\infty} q^{n^2} = 1+3 q+4 q^2+7 q^3+13 q^4+19 q^5+29 q^6+43 q^7 + \dots
\end{equation*}
Odd terms $f_{2n+1}$ stabilize to:
\begin{equation*}
        \frac{2\psi(q^2)}{f(-q)} = (q;q)_{\infty}^{-1} \sum_{n=-\infty}^\infty q^{n^2+n}
        = 2(1+q+3 q^2+4 q^3+7 q^4+10 q^5+17 q^6+23 q^7 + \dots ).
\end{equation*}

\subsection{Descendants} 

Garoufalidis and Kashaev defined in \cite[Def. 2.1]{GaroufalidisKashaev2022} the \emph{descendant Colored Jones function}:
\begin{equation}
	DJ_K^{(m)}(x) = \sum_{k=0}^\infty a_k \prod_{i=1}^k (x+x^{-1}-q^i-q^{-i}) q^{km},
\end{equation}
where $a_k$ are the (usual) Habiro coefficients.
Analogously, we can define a descendant version of the inverted Habiro series:
\begin{equation}
	DP_K^{(m)}(x) = \sum_{k=0}^\infty \frac{a_{-k-1}q^{k m}}{\prod_{i=0}^{k} (x+x^{-1}-q^i-q^{-i})}.
\end{equation}
For the knot $4_1$, the residues \[
\res_{x=q^j} DF_{4_1}^{(m)}(x) = -\sum_{k \geq |j|} \frac{(-1)^{k+j} q^{\binom{k+1}{2}+\binom{j+1}{2}} }{(q)_{k+j}(q)_{k-j}}.
\]
are quite similar to a series \cite[eq. 3.5]{Dimofte2011} studied in the context of hyperbolic geometry. For $j=0$, we have
	\[
		\res_{x=1} DF_{4_1}^{(m)}(x) = \sum_{k=0}^\infty (-1)^k \frac{q^{\binom{k+1}{2}+mk}}{(q)^2_k}.
	\]
which coincides with the series $G_0^{(m)}(q)$ from \cite{GaroufalidisZagier2023}.

\subsection{Nonabelian branches}

The $F_K$ series admits a conjectural generalization to a collection of series $F^{(\alpha)}_K$ associated with nonabelian branches $\alpha$ of the $A$-polynomial of the knot $K$ \cite{Ekholm2022}. The usual $F_K$ series corresponds to the abelian branch $F_K=F_K^{(\text{ab})}$. Park in \cite{ParkThesis} found the following inverted Habiro series\footnote{We added a normalization factor $x^{-1}-1$ for consistency with the rest of the paper.} for the two nonabelian branches $\alpha_{-1/2},\alpha_{1/2}$ of the knot $K=4_1$: 
\begin{equation}
	P_{4_1}^{(\alpha_{1/2})}(x) = e^{-\frac{(\log x)^2} {\log q}} (x^{-1}-1)\sum_{k \geq 0} \frac{q^{k^2}(q)^{-1}_k}{\prod_{i=0}^k(x+x^{-1}-q^i-q^{-i})}
\end{equation}
and
\begin{equation}
	P_{4_1}^{(\alpha_{-1/2})}(x) = e^{\frac{(\log x)^2} {\log q}} (x^{-1}-1)\sum_{k \geq 0} \frac{(-1)^kq^{-\binom{k}{2}}(q)^{-1}_k}{\prod_{i=0}^k(x+x^{-1}-q^i-q^{-i})}.
\end{equation}
We found a surprisingly simple relation between the residues of $P_{4_1}^{(\alpha_{1/2})}(x)$ and $P_{4_1}(x)$ at $x=q^j$ for $j \in \Z$:
\begin{equation}
	\res_{x=q^j} P_{4_1}^{(\alpha_{1/2})}(x) = 
   \frac{1}{(q)_\infty}\res_{x=q^j} P_{4_1}(x) = -\sum_{k \geq |j|} \frac{(-1)^{k+j}  q^{(3k^2 +k-j^2+j)/2}}{(q)_{k+j}(q)_{k-j}
   (q)_k}.
\end{equation}
In particular, the residue at $x=1$ corresponds to the series $G_0$ in \cite[eq. 13a]{Garoufalidis_2021}. 

For the other non-abelian branch $\alpha_{-1/2}$, we have:
\begin{equation}\label{eq:nonab_branch_minus}
	\res_{x=q^j} P_{4_1}^{(\alpha_{-1/2})}(x) =
    -\sum_{k \geq |j|}^\infty \frac{(-1)^j q^{k+j(3j+1)/2}}{(q)_{k+j}(q)_{k-j} (q)_k}.
\end{equation}

 Unlike for the residues of $P_{4_1}^{(\alpha_{1/2})}(x)$ and $P_{4_1}$, the series \eqref{eq:nonab_branch_minus} seems to have all coefficients of the same sign.

\section{Surgery formulas}\label{sec:surgery}

\subsection{GM surgery formula}
 
 Recall that $\Zhat_a(Y) = \Zhat_a(Y,q)$ is a (partially defined) invariant of a closed 3-manifold $Y$= equipped with a $\spc$ structure $a$. It is of the form
\[
    \Zhat_a(Y) = 2^{s} q^{\Delta_a} \sum_{n=0}^{\infty} c_n q^n
\]
with $\Delta_a \in \Q$, $s, c_n \in \Z.$ Let $Y=S^3_{p}(K)$ be the result of the integral Dehn surgery on a knot $K$ with surgery coefficient $p$. We represent a $\spc$ structure on $Y$ by an integer $a \in 0,1,\dots, \abs{p}-1$, using the (noncanonical) identification $\spc{S^3_{p}(K)} \cong H_1(S^3_p(K),\Z) \cong \Z/p\Z$ (see \cite[6.8]{GM19} for a careful treatment of $\spc$ structures).

Recall the GM surgery formula from \cite[p. 5, 46]{GM19} specialized to integer surgeries:
\begin{equation}\label{eq:GM_surgery}
	\Zhat_a(S^3_{p}(K)) \doteq \frac12 \mathcal{L}^{(a)}_{p}\left[ (x^{1/2}-x^{-1/2})  (x^{1/2}F_K(x) - x^{-1/2} F_K(x^{-1})) \right]
\end{equation}

The symbol $\mathcal{L}^{(a)}_{p}$ denotes a linear operator given on each monomial by
\[
	\mathcal{L}^{(a)}_{p} (x^uq^w) = 
	\begin{cases}
		q^{-u^2/p} \cdot q^w & \text{ if } u \equiv a \Mod{p}, \\
		0                      & \text{ otherwise.}        
	\end{cases}
\]

In terms of the coefficients $f_k$ (setting $f_{-1} \equiv 0$), \eqref{eq:GM_surgery} reads:
\begin{equation}\label{eq:surg_from_fk}
\begin{split}
\Zhat_a(S^3_{p}(K)) 
    &\doteq \frac12 \mathcal{L}^{(a)}_{p} \sum_{k \geq 0} f_k (x^{k+1}+x^{-k-1}-x^k-x^{-k})\\
    &\doteq \frac12 \mathcal{L}^{(a)}_{p} \sum_{k \geq 0} (f_{k-1}-f_k) (x^k+x^{-k})\\
    &  \doteq \frac12 \sum_{
            \substack{
            k \equiv \pm a \! \Mod{p}\\
            k \geq 0}
            } 
    q^{-k^2/p}(f_{k-1} - f_{k}).
\end{split}
\end{equation}
Here we write $\doteq$ for an equality up to a sign and a rational power of $q$ (denoted by $\epsilon q^d$ in \cite[Thm.\ 1.2]{GM19}), not essential for our discussion.
%

A necessary condition for the validity of the surgery formula is the convergence of \eqref{eq:surg_from_fk}, which depends on the growth of the minimal degrees $\delta(f_k)$. Since the inverted Habiro series `regularizes' the negative powers in $f_k$, it is plausible that it should lead to an extension of the surgery formula.

\subsection{Surgery using residues}

We can express \eqref{eq:surg_from_fk} in terms of the residues $r_j$ of $P_K(x)$ using \cref{lem:fk_from_residues}:
\begin{equation}\label{eq:surg_from_fk_coefs}
    \Zhat_a(S^3_{p}(K)) \doteq \frac12 \sum_{
            \substack{k \equiv \pm a \Mod{p}\\ k \geq 0}
            } 
    q^{-k^2/p} \sum_{j \geq 1}(q^{jk}-q^{-jk}+q^{-j(k+1)}-q^{j(k-1)}) r_j.
\end{equation}

Depending on the growth of $\delta(r_j)$, there is a certain range of the coefficients $p$ for which \eqref{eq:surg_from_fk_coefs} converges as a \emph{double} sum. 

\begin{proposition}
    If \eqref{eq:surg_from_fk_coefs} converges as a double sum for some $p$, then 
    \begin{equation}\label{eq:surg_from_residues}
    \Zhat_a(S^3_{p}(K)) 
    \doteq \sum_{j \geq 1}r_j (1-q^{-j})  \sum_{n=0}^{j-1}q^{j(n p + a)-\frac{1}{p}(n p + a)^2}.
\end{equation}
In particular, the coefficient of each $r_j$ is a Laurent polynomial.
\end{proposition} 

\begin{proof}
We can switch the order of the summation and express the surgery formula as 
\begin{equation*}
\begin{split}
    \Zhat_a(S^3_{p}(K)) &\doteq  \frac12 \sum_{j \geq 1}r_j (1-q^{-j})
        \sum_{
            \substack{k \equiv \pm a \Mod{p}\\ k \geq 0}
            }
        (q^{jk}-q^{-jk})q^{-k^2/p} \\
      &\doteq  \frac12 \sum_{j \geq 1}r_j (1-q^{-j})
       \left(\sum_{n=0}^{j-1}q^{j(n p + a)-\frac{(n p + a)^2}{p} }+  \sum_{n=0}^{j-1}q^{j(n p + p- a)-\frac{(n p + p - a)^2}{p}}\right).
\end{split}
\end{equation*}
In the second equation, we set $k=np+a$ (resp. $k=np+p-a$) and performed a telescopic sum using
\[
-j(n p+a)-\frac{(n p+a)^2}{p} =-\frac{(np+a)((n + j)p+a)}{p}= j((n+j)p+a)-\frac{((n+j)p+a)^2}{p}.
\]
To obtain \eqref{eq:surg_from_residues}, we use the symmetry $n \to j-1-n$ to identify the two sums.
\end{proof}

Using \eqref{eq:res_series}, we can express \eqref{eq:surg_from_residues} in terms of the inverted Habiro coefficients: 
\begin{equation}\label{eq:surgery_from_ih_coef}
    \Zhat_a(S^3_{p}(K)) 
    \doteq \sum_{k \geq 1}{a_{-k-1}} \sum_{j=1}^k \sum_{n=0}^{j-1}\frac{(-1)^{k+j} q^{\binom{k+1}{2}+\binom{j+1}{2}} (1-q^{-j}) q^{j(n p + a)-\frac{(n p + a)^2}{p}}}{(q)_{k+j}(q)_{k-j}}
\end{equation}

For $p<0$, the derivation of \eqref{eq:surg_from_residues} from \eqref{eq:GM_surgery} is equivalent to an application of the residue theorem for $P(x) \Theta^{p,a}(x)$, where 
\begin{equation}\label{eq:theta_surgery}
    \Theta^{p,a}(x)=\sum_{{\substack{u<0\\ u \equiv a \Mod{p}}}} x^u q^{-\frac{u^2}{p}}.
\end{equation}
The GM formula \eqref{eq:GM_surgery} corresponds to the residue at $x=0$ while \eqref{eq:surg_from_residues} corresponds to the sum of other residues. This was implicitly done by Park, who formulated the $-p$-surgery formula for $p>0$ using certain polynomials\footnote{Not to be confused with the inverted Habiro series $P_K(x)$.} $P^{p,a}_k(q)$ \cite[eq. 27]{Park21}. Their definition can be formulated as the following residue:
\begin{equation}\label{eq:parks_poly_res}
    P^{p,a}_k(q) = q^{-k^2-\frac{a(p-k)}{p}} (q^{k+1};q)_k  \res_{x=0} \left(\frac{1}{\prod_{i=1}^k (x+x^{-1}-q^i-q^{-i})}   \Theta^{-p,a}(x)\right).
\end{equation}

By comparing this with \eqref{eq:surgery_from_ih_coef} or applying the residue theorem directly to \eqref{eq:parks_poly_res}, we obtain an explicit formula for the polynomials $P^{p,a}_k(q)$, conjectured by M. Jagadale:
\begin{proposition}\label{prop:p-surgery}
    \begin{equation*}
    P_k^{p,a}(q) =
    - q^{\frac{a (p-a)}{p}} (q^{k+1};q)_k \sum_{j=1}^k \frac{(-1)^{k+j} (1-q^{-j}) q^{\binom{j+1}{2}-\binom{k}{2}}}{(q)_{k+j}(q)_{k-j}}
   \sum_{n=0}^{j-1} q^{\frac{(n p+a)^2}{p}-j (np+a)}.
    \end{equation*}
\end{proposition}

For $p$-surgery with $p>0$, the residue theorem cannot be used, as the theta function in \eqref{eq:theta_surgery} does not converge for $|q|<1$. Nevertheless, the formula \eqref{eq:surg_from_fk_coefs} can be taken as an extension of the GM surgery formula, and it coincides with Park's `regularized' surgery formula \cite[Conj.\ 4]{Park20}. Therefore, it unifies both the GM formula and Park's formula, at least in a certain range. The difference between them GM and Park's formulas can be viewed as a failure of Fubini's theorem in the double sum \eqref{eq:surg_from_fk_coefs}.  

\bibliographystyle{abbrv}
\bibliography{bibliography}

\end{document}